\documentclass[11pt,a4paper]{article}

\usepackage{amsfonts}
\usepackage{mathrsfs}
\usepackage{comment}
\usepackage{amsmath}
\usepackage{amssymb}
\usepackage{amsthm}
\usepackage{amscd}
\usepackage[all]{xy}
\usepackage{titlesec}
\usepackage[top=20mm, bottom=20mm, left=20mm, right=20mm]{geometry}
\usepackage{color}

\newtheorem{theorem}{Theorem}[]
\newtheorem{lemma}[theorem]{Lemma}
\newtheorem{proposition}[theorem]{Proposition}

\newtheorem{conjecture}[theorem]{Conjecture}

\newtheorem{problem}[theorem]{Problem}

\newcommand{\ma}{\mathcal}

\newcommand{\mr}{\mathscr}
\newcommand{\s}{\subseteq}

\newcommand{\fr}{\frac}
\newcommand{\lc}{\lceil}
\newcommand{\rc}{\rceil}
\newcommand{\lf}{\lfloor}
\newcommand{\rf}{\rfloor}

\newcommand{\ep}{\epsilon}
\newcommand{\ex}{{\rm{E}}}
\newcommand{\e}{{\rm{ex}}}

\begin{document}
\title{Sparse hypergraphs with applications to coding theory\footnote{Part of this paper has been published in 2019 IEEE International Symposium on Information Theory.}}

\author{Chong Shangguan\footnote{Department of Electrical Engineering-Systems, Tel Aviv University, Tel Aviv 6997801, Israel. Email: theoreming@163.com.} and Itzhak Tamo\footnote{Department of Electrical Engineering-Systems, Tel Aviv University, Tel Aviv 6997801, Israel. Email: zactamo@gmail.com.}}

\date{}

\maketitle

\begin{abstract}
  For fixed integers $r\ge 3,e\ge 3,v\ge r+1$, an $r$-uniform hypergraph is called $\mathscr{G}_r(v,e)$-free if the union of any $e$ distinct edges contains at least $v+1$ vertices.
  Brown, Erd\H{o}s and S\'{o}s showed that the maximum number of edges of such a hypergraph on $n$ vertices, denoted as $f_r(n,v,e)$, satisfies
  $$\Omega(n^{\fr{er-v}{e-1}})=f_r(n,v,e)=O(n^{\lceil\frac{er-v}{e-1}\rceil}).$$
  For sufficiently large $n$ and $e-1\mid er-v$, the lower bound matches the upper bound up to a constant factor, which depends only on $r,v,e$; whereas for $e-1\nmid er-v$, in general it is a notoriously hard problem to determine the correct exponent of $n$.
  Among other results, we improve the above lower bound by showing that
  $$f_r(n,v,e)=\Omega(n^{\frac{er-v}{e-1}}(\log n)^{\frac{1}{e-1}})$$
  for any $r,e,v$ satisfying $\gcd(e-1,er-v)=1$.
  The hypergraph we constructed is in fact $\mathscr{G}_r(ir-\lceil\frac{(i-1)(er-v)}{e-1}\rceil,i)$-free for every $2\le i\le e$, and it has several interesting applications in coding theory.
  The proof of the new lower bound is based on a novel application of the lower bound on the hypergraph independence number due to Duke, Lefmann, and R{\"o}dl.
\end{abstract}



\section{Introduction}

\noindent Since the pioneering work of Tur\'an \cite{turan}, the study of Tur\'an-type problems has been playing a central role in the field of extremal combinatorics. In this work, we present an improved probabilistic lower bound for a hypergraph Tur\'an-type problem introduced by Brown, Erd\H{o}s and S\'os \cite{BES71} in 1973. We also show that this new bound provides improved constructions for several seemingly unrelated problems in coding theory, including Parent-Identifying Set Systems, uniform Combinatorial Batch Codes and optimal Locally Recoverable Codes.

Let us begin with some necessary notation. For an integer $r\ge 2$, an $r$-uniform hypergraph (henceforth an $r$-graph) $\ma{H}:=(V(\ma{H}),E(\ma{H}))$ can be viewed as a pair of vertices and edges, where the vertex set $V(\ma{H})$ is a finite set and the edge set $E(\ma{H})$ is a collection of $r$-subsets of $V(\ma{H})$. An $r$-graph is called {\it $\ma{H}$-free} if it contains no subhypergraph which forms a copy of $\ma{H}$. For a family $\mr{H}$ of $r$-graphs, the {\it Tur\'an number}, $\e_r(n,\mr{H})$, is the maximum number of edges in an $r$-graph on $n$ vertices which is $\ma{H}$-free for every $\ma{H}\in\mr{H}$.

Throughout this paper, an $r$-graph $\ma{H}$ always stands for its edge set $E(\ma{H})$. The vertex set $V(\ma{H})$ is viewed as a subset of $[n]:=\{1,\ldots,n\}$. Given a finite set $X\s[n]$, denote by $\binom{X}{r}$ the family of $\binom{|X|}{r}$ distinct $r$-subsets of $X$. Hence, $\ma{H}=E(\ma{H})\s\binom{[n]}{r}$. We will frequently use the standard Bachmann-Landau notations $\Omega(\cdot),\Theta(\cdot),O(\cdot)$ and $o(\cdot)$, whenever the constants are not important.

For integers $r\ge 2,e\ge 2,v\ge r+1$, let $\mr{G}_r(v,e)$ be the family of all $r$-graphs formed by $e$ edges and at most $v$ vertices; that is,
$$\mr{G}_r(v,e)=\{\ma{H}\s\binom{[v]}{r}:|\ma{H}|=e\}.$$
An $r$-graph $\ma{H}$ is called {\it $\mr{G}_r(v,e)$-free} if it does not contain a copy of any member of $\mr{G}_r(v,e)$, namely, the union of any $e$ distinct edges of $\ma{H}$ contains at least $v+1$ vertices. In the literature, such $r$-graphs are also termed {\it sparse} \cite{sparse}. As in the previous papers (see, e.g. \cite{AlonShapira}), we use the notation $f_r(n,v,e):=\e_r(n,\mr{G}_r(v,e))$.

Since the study of $f_r(n,v,e)$ for $e=2$ or $r=2$ has been quite extensive (see, e.g. \cite{Erdos1964,Erdosr=2,Rodlpacking}), we focus on the asymptotic behavior of $f_r(n,v,e)$ for fixed integers $r\ge 3,e\ge 3,v\ge r+1$ as $n\rightarrow\infty$.
It was shown in \cite{BES71} that in general
\begin{equation}\label{BESbound}
  \begin{aligned}
   \Omega(n^{\fr{er-v}{e-1}})=f_r(n,v,e)=O(n^{\lc\fr{er-v}{e-1}\rc}).
  \end{aligned}
\end{equation}

\noindent The lower bound in (\ref{BESbound}) is obtained by a standard probabilistic method (now known as the alteration method, see, e.g. \cite{alon2016probabilistic}), and the (naivest) upper bound follows from a double counting argument, which uses the simple fact that any set of $\lc\fr{er-v}{e-1}\rc$ vertices can be contained in at most $e-1$ distinct edges.

Observe that the exponent of $n$ in (\ref{BESbound}) is tight for $e-1\mid er-v$; however, for $e-1\nmid er-v$, in general it is a notoriously hard problem to determine the correct order of the exponent of $n$. In particular, for fixed $r>k\ge 2,e\ge 3$ and $v=e(r-k)+k+1$, the study of $f_r(n,e(r-k)+k+1,e)$ as $n\rightarrow\infty$ has attracted considerable attention since the work of \cite{BES71,MR323647}. It is easy to check by (\ref{BESbound}) that
\begin{equation}\label{previouslowerbound}
\begin{aligned}
  \Omega(n^{k-\fr{1}{e-1}})=f_r(n,e(r-k)+k+1,e)=O(n^k).
\end{aligned}
\end{equation}

The following conjecture remains widely open.
\begin{conjecture}[see, \cite{BES71,AlonShapira}]\label{conjecture}
  For fixed integers $r>k\ge 2, e\ge 3$,
  $$n^{k-o(1)}<f_r(n,e(r-k)+k+1,e)=o(n^k)$$ as $n\rightarrow\infty$.
\end{conjecture}

Conjecture \ref{conjecture} has been studied in depth for more than forty years. For example, the first case of the conjecture, i.e., when $r=3, k=2$ and $e=3$, was already highly nontrivial. It was not solved until Ruzsa and Szemer\'{e}di \cite{Ruzsa-Szemeredi} proved the (6,3)-theorem
$$n^{2-o(1)}<f_3(n,6,3)=o(n^2),$$
where the upper bound follows from the celebrated Regularity Lemma \cite{Szemeredi}, and the lower bound is based on Behrend's construction \cite{Behrend46} on 3-term arithmetic progression free sets. The study of $f_3(n,6,3)$ indicates that the resolution of Conjecture \ref{conjecture} may rely heavily on the regularity lemmas\footnote{which include, for example, the graph regularity lemma and the hypergraph regularity lemma, see, \cite{Conlon-Fox}} and Behrend-type constructions, which are among the most powerful tools in extremal combinatorics. Improvements of (\ref{BESbound}) on sporadic or less general parameters have been obtained in a line of other works \cite{EPR,AlonShapira,other2,other1,Rodl,ge2017sparse}. Currently, the upper bound part of Conjecture \ref{conjecture} is known to be true for all $r\ge k+1\ge e\ge 3$ \cite{ge2017sparse}, and the lower bound part holds for $k>r\ge 2$ \cite{AlonShapira} and $k=2,e\in\{4,5,7,8\}$ \cite{ge2017sparse}.

Despite the efforts of many researchers, the lower bound \eqref{previouslowerbound} implied by (\ref{BESbound}) remains the best possible for $e\ge 4, r>k\ge 3$ and $e\not\in\{3,4,5,7,8\}, r>k=2$. In the proposition below we slightly improve the lower bound of \eqref{previouslowerbound} by a $(\log n)^{\fr{1}{e-1}}$ factor.

\begin{proposition}\label{presentlowerbound}
  For fixed integers $r>k\ge 2, e\ge 3$, 
  $$f_r(n,e(r-k)+k+1,e)=\Omega(n^{k-\fr{1}{e-1}}(\log n)^{\fr{1}{e-1}})$$
  as $n\rightarrow\infty$.
\end{proposition}

Proposition \ref{presentlowerbound} is in fact an easy consequence of the following more general result.

\begin{theorem}\label{sparseturannumber1}
  For fixed integers $r\ge 3,e\ge 3, v\ge r+1$ satisfying $\gcd(e-1,er-v)=1$ and sufficiently large $n$, there exists an $r$-graph with
  $$\Omega(n^{\fr{er-v}{e-1}}(\log n)^{\fr{1}{e-1}})$$
  edges, which is simultaneously $\mr{G}_r(ir-     \lceil \fr{(i-1)(er-v)}{e-1}\rceil    ,i)$-free for every $2\le i\le e$. In particular, setting $i=e$ we have
  $$f_r(n,v,e)=\Omega(n^{\fr{er-v}{e-1}}(\log n)^{\fr{1}{e-1}}).$$  
\end{theorem}

The proof of this theorem is presented in Section \ref{maintheorem}. To see that Proposition \ref{presentlowerbound} indeed follows from  Theorem \ref{sparseturannumber1}, it suffices to write $er-v=k(e-1)+q$ for some $1\le q\le e-2$ (we exclude $q=0$ since in that case the exponent of (\ref{BESbound}) is tight), then $\gcd(e-1,er-v)=1$ holds, for example, when $q=\pm1$ or $e-1$ is a prime. Proposition \ref{presentlowerbound} follows by setting  $q=-1$.

The proof of Theorem \ref{sparseturannumber1} relies on a novel application of the lower bound on the hypergraph independence number due to Duke, Lefmann, and R{\"o}dl \cite{indset2}, as stated in Section \ref{maintheorem}. Since in our proof we cannot get rid of the coprime condition, it remains an interesting question to determine whether this constraint is necessary. Moreover, we have the following open problem.

\begin{problem}\label{conjecture2}
  For which parameters $r,e,v$ satisfying $r\ge 3, e\ge 3, v\ge r+1$ and $e-1\nmid er-v$, there exist a constant $\epsilon>0$ such that for sufficiently large $n$, $$f_r(n,v,e)=\Omega(n^{\fr{er-v}{e-1}+\epsilon})?$$
\end{problem}

It is noteworthy that sparse hypergraphs have found many applications in theoretical computer science and coding theory, some of which are listed as follows:

\begin{itemize}
  \item $\mr{G}_3(6,3)$-free 3-graphs were used in PCP analysis and Linearity Testing \cite{PCPtesting}, Communication Complexity \cite{CommunicationComplexity}, Monotonicity Testing \cite{Monotonicitytesting} and Coded Caching Schemes \cite{codedcaching};

  \item $\mr{G}_r(er-r,e)$-free $r$-graphs can be used to construct Perfect Hash Families \cite{Shangguan16} and Parent-Identifying Set Systems (IPPSs for short);

  \item $r$-graphs which are simultaneously $\mr{G}_r(i-1,i)$-free for each $1\le i\le e$
were used to construct uniform Combinatorial Batch Codes \cite{batch1} (uniform CBCs for short);
and in particular, for $r=3$ they were used in a bitprobe model with three probes \cite{probes};

\item  $r$-graphs which are simultaneously $\mr{G}_r(ir-i,i)$-free for each $1\le i\le e$ were used to construct
optimal Locally Recoverable Codes \cite{LRC} (optimal LRCs for short).
\end{itemize}
In Section \ref{application-2} we will present the applications of Theorem \ref{sparseturannumber1} in the constructions of IPPSs, uniform CBCs and optimal LRCs.

The rest of this paper is organized as follows. 
In Section \ref{maintheorem} we present the proof of our main result, namely Theorem \ref{sparseturannumber1}.
In Section \ref{application-1} we discuss the applications of Theorem \ref{sparseturannumber1} to two problems in extremal combinatorics, and in Section \ref{application-2} we present three applications of Theorem \ref{sparseturannumber1} to coding theory.

\section{Proof of the main result} \label{maintheorem}

\noindent To prove Theorem \ref{sparseturannumber1} we will make use of the following lemma of Duke, Lefmann, and R{\"o}dl \cite{indset2} (whose proof applied a result of \cite{inpset1}). Note that an {\it independent set} of an $r$-graph is a subset of vertices such that no $r$ elements form an edge, and an $r$-graph is said to be {\it linear} if any two distinct edges share at most one vertex.

\begin{lemma}[see Theorem 2, {\rm{\cite{indset2}}}]\label{indset}
  For all fixed $r\ge3$ there exists a constant $c>0$ depending only on $r$ such that every linear $r$-graph on $n$ vertices with average degree\footnote{The original theorem in \cite{indset2} has the condition `with maximum degree at most $d$'. However, since for any hypergraph with average degree at most $d$, there exists a subhypergraph of it which has at least half of its vertices and maximum degree at most $2d$, it is not hard to observe that the assertion of the original theorem works also with the condition `with average degree at most $d$', at the expense of a worse constant $c$.} at most $d$ has an independent set of size at least $cn(\fr{\log d}{d})^{\fr{1}{r-1}}.$
\end{lemma}

Recall that we view the parameters $v,e,r$ as constants, whereas $n$ tends to infinity. Since we are only interested in the asymptotic behavior we do not make an attempt to optimize any of the constants. The following two inequalities are well known (see, e.g. \cite{alon2016probabilistic}).

\vspace{5pt}

\noindent\textbf{Chernoff's inequality.}  Suppose $X_1,\ldots,X_n$ are independent random variables taking values in $\{0,1\}$. Let $X$ denote their sum and let $\mu=\ex[X]$ denote the sum's expected value. Then for any $\delta\in[0,1]$, $\Pr[X\le(1-\delta)\mu]\le e^{-\fr{\delta^2\mu}{2}}$.

\vspace{5pt}

\noindent\textbf{Markov's inequality.} If $X$ is a nonnegative random variable and $a>0$, then $\Pr[X\ge a]\le\fr{\ex[X]}{a}$.

\vspace{5pt}

Below we present the proof of Theorem \ref{sparseturannumber1}.

\begin{proof}[\textbf{Proof of Theorem \ref{sparseturannumber1}}]
Set $p:=p(n)=\Theta(n^{-\fr{v-r}{e-1}+\ep})$ for some $\ep>0$, which will be made explicit later.
Generate an $r$-graph $\ma{H}_0\s\binom{[n]}{r}$ by picking each member of $\binom{[n]}{r}$ independently with probability $p$.
Let $X$ denote the number of edges in $\ma{H}_0$.
Clearly,
   \begin{equation}\label{X}
     \begin{aligned}
       \ex[X]=p\binom{n}{r}=\Theta(n^{\fr{er-v}{e-1}+\ep}).
     \end{aligned}
   \end{equation}

For  $2\le i\le e-1$, let $\mr{Y}_i$ be the collection of all $i$ distinct edges of $\ma{H}_0$
whose union contains at most $ir-f(i)$ vertices, where  $f(i)$ will be determined later.
Let $Y_i$ denote the size of $\mr{Y}_i$. Then
   \begin{equation}\label{Yi}
      \begin{aligned}
       \ex[Y_i]=O(p^in^{ir-f(i)})=O(n^{\fr{i(er-v)}{e-1}+i\ep-f(i)}),
     \end{aligned}
   \end{equation}

\noindent where the first equality follows from the fact that there are at most $O(n^{ir-f(i)})$ ways to choose $i$ edges whose union contains at most $ir-f(i)$ vertices.

We say that $e$ distinct edges of $\ma{H}_0$ form a bad $e$-system if their union contains at most $v$ vertices. Clearly, two distinct bad $e$-systems can share at most $e-1$ edges.
For each $2\le i\le e-1$, let $\mr{Z}_i$ be the collection of the unordered pairs of bad $e$-systems which share precisely $i$ edges, and the union of those $i$ common edges contains at least $ir-f(i)+1$ vertices.
For $\{\ma{Z},\ma{Z}'\}\in \mr{Z}_i$, it is clear that $|E(\ma{Z})\cup E(\ma{Z}')|=2e-i$ and $|V(\ma{Z})\cup V(\ma{Z}')|\le 2v-(ir-f(i)+1).$
Let $Z_i$ denote the size of $\mr{Z}_i$. Then
   \begin{equation}\label{Zi}
      \begin{aligned}
       \ex[Z_i]=O(p^{2e-i}n^{2v-(ir-f(i)+1)})=O(n^{f(i)+\ep(2e-i)-\fr{(i-2)(er-v)}{e-1}-1}),
     \end{aligned}
   \end{equation}

\noindent where the first equality follows from the fact that there are at most $O(n^{2v-(ir-f(i)+1)})$ ways to choose $2e-i$ edges whose union contains at most $2v-(ir-f(i)+1)$ vertices.
Lastly, let $W$ denote the number of bad $e$-systems in $\ma{H}_0$. Then
   \begin{equation}\label{W}
     \begin{aligned}
       \ex[W]=O(p^en^v)=O(n^{\fr{er-v}{e-1}+e\ep}).
     \end{aligned}
   \end{equation}

   In order to apply Lemma \ref{indset}, we will bound from above the number of pairs of bad $e$-systems which share at least two edges, by picking $\epsilon$ and $f(i)$ so that
   \begin{equation*}\label{con}
       \ex[Y_i]=o(\ex[X])\text{\quad and\quad}\ex[Z_i]=o(\ex[X])
   \end{equation*}
   for each $2\le i\le e-1$ as $n\rightarrow\infty$.
   From (\ref{X}), (\ref{Yi}) and (\ref{Zi}), it is easy to see that $\ex[Y_i]=o(\ex[X])$ if and only if
   \begin{equation}\label{condition1}
     f(i)>\ep(i-1)+\fr{(i-1)(er-v)}{e-1},
   \end{equation}
   and $\ex[Z_i]=o(\ex[X])$ if and only if
   \begin{equation}\label{condition2}
     f(i)<\fr{(i-1)(er-v)}{e-1}-\ep(2e-i-1)+1.
   \end{equation}
      Let $ a=\min_{2\le i\le e-1}\left\{\fr{1}{i-1}\big(f(i)-\fr{(i-1)(er-v)}{e-1}\big),~\fr{1}{2e-i-1}\big(\fr{(i-1)(er-v)}{e-1}+1-f(i)\big)\right\}.$
   There is $\epsilon\in(0,a)$ 
   satisfying
  (\ref{condition1}) and (\ref{condition2}) if and only if for each $2\le i\le e-1$, there exists an integer $f(i)$ such that
 \begin{equation} \label{f(i)}
   \begin{aligned}
     \fr{(i-1)(er-v)}{e-1}<f(i)<\fr{(i-1)(er-v)}{e-1}+1.
   \end{aligned}
   \end{equation}
   Since $f(i)$ is an integer, (\ref{f(i)}) holds if and only if $e-1\nmid(i-1)(er-v).$
   It is easy to verify that those $e-2$ indivisibility conditions hold simultaneously if and only if $\gcd(e-1,er-v)=1.$
   Under this condition, it suffices to pick for each $2\le i\le e-1$,
   $$f(i)=\lceil \fr{(i-1)(er-v)}{e-1}\rceil$$ and an arbitrary $\ep\in(0,a)$
   (note that by the choices of $f(i)$ we have $a>0$).

   Applying Chernoff's inequality for $X$ and Markov's inequality for $Y_i,Z_i$ and $W$, it is easy to see that for each $2\le i\le e-1$ and sufficiently large $n$,
   $$\Pr[X<0.9\ex[x]]<\fr{1}{2e},~\Pr[Y_i>2e\ex[Y_i]]<\fr{1}{2e},~\Pr[Z_i>2e\ex[Z_i]]<\fr{1}{2e},~\Pr[W>2e\ex[W]]<\fr{1}{2e}.$$
   Therefore, with positive probability, there exists an $r$-graph $\ma{H}_0\s\binom{[n]}{r}$ such that for each $2\le i\le e-1$,
   \begin{equation}\label{eq-new}
     X\ge 0.9 \ex[X], ~Y_i\le 2e\ex[Y_i],~Z_i\le 2e\ex[Z_i],~W\le 2e\ex[W].
   \end{equation}

   Fix such an $\ma{H}_0$. We construct a subhypergraph $\ma{H}_1$ of $\ma{H}_0$ as follows. For every $2\le i\le e-1$, remove from $\ma{H}_0$ one edge from each member of $\mr{Y}_i$, and one edge from $E(\ma{Z})\cup E(\ma{Z}')$ for each pair $\{\ma{Z},\ma{Z}'\}\in\mr{Z}_i$. It is not hard to check that $\ma{H}_1$ satisfies the following properties:
   \begin{itemize}
     \item [(i)] $|\ma{H}_1|=\Omega(n^{\fr{er-v}{e-1}+\ep})$;

     \item [(ii)] the number of bad $e$-systems contained in $\ma{H}_1$ is at most $O(p^en^v)=O(n^{\fr{er-v}{e-1}+e\ep})$;

     \item [(iii)] for each $2\le i\le e-1$, the union of any $i$ distinct edges in $\ma{H}_1$ contains more than $ir-f(i)$ vertices;
     \item [(iv)] any two bad $e$-systems in $\ma{H}_1$ can share at most one edge.
   \end{itemize}

Indeed, (i) is an easy consequence of the following calculation:
\begin{equation*}
  |\ma{H}_1|\ge|\ma{H}_0|-\sum_{i=2}^{e-1}|\mr{Y}_i|-\sum_{i=2}^{e-1}|\mr{Z}_i|\ge 0.9\ex[X]-o(\ex[X])=\Omega(\ex[X]);
\end{equation*}
\noindent (ii) follows from \eqref{X}, (\ref{W}) and the observation that removing edges from $\ma{H}_0$ does not increase the number of bad $e$-systems; (iii) holds since according to our construction, $\ma{H}_1$ does not contain any member of $\mr{Y}_i$ for any $2\le i\le e-1$.
It remains to verify (iv).
Assume to the contrary that $\ma{H}_1$ still contains two bad $e$-systems that share $i$ edges for some $2\le i\le e-1$.
On one hand, if those $i$ edges are spanned by at least $ir-f(i)+1$ vertices, then the pair of such two bad $e$-systems must belong to $\mr{Z}_i$, which is a contradiction. 
On the other hand, if those $i$ edges are spanned by at most $ir-f(i)$ vertices, then they must form a member of $\mr{Y}_i$, which is again a contradiction. 

Next we construct an auxiliary $e$-graph $\ma{U}\s\binom{E(\ma{H}_1)}{e}$ as follows:
   \begin{itemize}
     \item  the vertex set of $\ma{U}$ is formed by the edge set of $\ma{H}_1$;
     \item  $e$ vertices of $\ma{U}$ form an edge if and only if the corresponding $e$ $r$-edges in $\ma{H}_1$ form a bad $e$-system.
        \end{itemize}

     \noindent It is routine to check that the following hold:
     \begin{itemize}
       \item   $\ma{U}$ is linear (by (iv));
       \item  $\ma{U}$ has at least $\Omega(n^{\fr{er-v}{e-1}+\ep})$ vertices (by (i)) and at most $O(p^en^v)=O(n^{\fr{er-v}{e-1}+e\ep})$ edges (by (ii));
       \item  $d(\ma{U})$, the average degree of $\ma{U}$, is at most
       $d(\ma{U})=O(\fr{e\cdot |E(\ma{U})|}{|V(\ma{U})|})=O(n^{(e-1)\ep}).$
     \end{itemize}

     \noindent Lemma \ref{indset} therefore applies and $\ma{U}$ has an independent set of size at least
     \begin{equation*}\label{independent}
       \Omega\big(|V(\ma{U}|\cdot(\fr{\log d(\ma{U})}{d(\ma{U})})^{\fr{1}{e-1}}\big)=\Omega(n^{\fr{er-v}{e-1}}(\log n)^{\fr{1}{e-1}}).
     \end{equation*}

     Now the theorem follows from the following simple observation:
     every independent set $\ma{I}\s V(\ma{U})$ corresponds to a $\mr{G}_r(v,e)$-free subhypergraph $\ma{H}_{\ma{I}}\s\ma{H}_1$ with $|\ma{I}|$ edges; moreover, by (iii)
     $\ma{H}_{\ma{I}}$ is also $\mr{G}_r(ir-f(i),i)$-free for each $2\le i\le e-1$.
\end{proof}

The proof of Theorem \ref{sparseturannumber1} leads to the following proposition.

\begin{proposition}
    Let $s\ge 1, r\ge 3$ and $(v_i,e_i),1\le i\le s$ be fixed integers satisfying $v_i\ge r+1,e_i\ge 2$. Suppose further that $e_1\ge 3, \gcd(e_1-1,e_1r-v_1)=1$ and $\frac{e_1r-v_1}{e_1-1}<\frac{e_ir-v_i}{e_i-1}$ for $2\le i\le s$.
    Then there exists an $r$-graph with $\Omega(n^{\frac{e_1r-v_1}{e_1-1}}(\log n)^{\frac{1}{e_1-1}})$ edges which is $\mr{G}_r(v_i,e_i)$-free for each $1\le i\le s$.
\end{proposition}

\begin{proof}[Sketch of the proof]
With the notation of the previous proof, we generate an $r$-graph $\ma{H}_0\subseteq\binom{[n]}{r}$ by picking each element of $\binom{[n]}{r}$ independently with probability $p:=\Theta(n^{-\frac{v_1-r}{e_1-1}+\epsilon})$ for some small constant $\epsilon>0$. For $2\le i\le e$, let $X,Y_i,Z_i,W,f(i)$ and $a\in(0,1)$ be defined analogously to the proof of Theorem \ref{sparseturannumber1} but with respect to $v_1$ and $e_1$. Hence, the expected number of edges in $\ma{H}_0$ is $\ex[X]=\Theta(n^{\frac{e_1r-v_1}{e_1-1}+\epsilon})$.
Moreover, for $2\le j\le s$, the expected number of bad $e_j$-systems contained in $\ma{H}_0$ is $\ex[W_j]=\Theta(n^{v_j-\frac{e_j(v_1-r)}{e_1-1}+e_j\epsilon})$.

By assumption, it is clear that $\frac{v_j-r}{e_j-1}<\frac{v_1-r}{e_1-1}$ for each $2\le j\le s$. Let $b=\min_{2\le j\le s}~\{\frac{v_1-r}{e_1-1}-\frac{v_j-r}{e_j-1}\}$. Then for any $\epsilon\in(0,b)$ and $2\le j\le s$,
 \begin{equation*}
     (\frac{e_1r-v_1}{e_1-1}+\epsilon)-(v_j-\frac{e_j(v_1-r)}{e_1-1}+e_j\epsilon)=(e_j-1)(\frac{v_1-r}{e_1-1}-
     \frac{v_j-r}{e_j-1}-\epsilon)>0.
 \end{equation*}

 \noindent Choosing an arbitrary $\epsilon\in(0,\min\{a,b\})$, it is easy to check that for any $2\le i\le e$ and $2\le j\le s$,
 $$\ex[Y_i]=o(\ex[X]),~\quad\ex[Z_i]=o(\ex[X])\text{\quad and\quad}\ex[W_j]=o(\ex[X]).$$

 \noindent Similar to the previous proof, by applying Chernoff's inequality for $X$ and Markov's inequality for $Y_i,Z_i,W$ and $W_j$, one can show that with positive probability there exists an $r$-graph $\ma{H}_0\s\binom{[n]}{r}$ such that 
   \begin{equation*}
     X=\Omega(\ex[X]),~W=O(\ex[W]),~\Gamma=o(\ex[X])\text{~for every $\Gamma\in\{Y_i,Z_i,W_j:2\le i\le e,2\le j\le s\}$}.
   \end{equation*}

The rest of the proof follows fairly straightforwardly from the argument of the previous proof, hence is omitted.
\end{proof}

\section{Applications to two extremal problems}\label{application-1}

\noindent The probabilistic construction of Theorem \ref{sparseturannumber1} immediately implies new lower bounds for two hypergraph extremal problems, as stated below.

\subsection{$\mr{H}_r(q,e)$-free $r$-graphs}

\noindent Bujt\'{a}s and Tuza \cite{BujtasTuzaCBC} studied the following extremal problem which is related to the construction of uniform Combinatorial Batch Codes (see Subsection \ref{application-CBC} below).
An $r$-graph is said to be {\it $\mr{H}_r(q,e)$-free} if it is simultaneously $\mr{G}_r(i-q-1,i)$-free for every $1\le i\le e$.
In \cite{BujtasTuzaCBC} it was shown that for fixed integers $r\ge 3, e> q+r\ge 3$, $$\e_r(n,\mr{H}_r(q,e))=\Omega(n^{r-1+\fr{r+q}{e-1}}).$$

The following proposition is a direct consequence of Theorem \ref{sparseturannumber1}.

\begin{proposition} \label{universeconsequence1}
  For fixed integers $r\ge 3, e> q+r\ge 3$ with $\gcd(e-1,r+q)=1$,
  $$\e_r(n,\mr{H}_r(q,e))=\Omega(n^{r-1+\fr{r+q}{e-1}}(\log n)^{\fr{1}{e-1}})$$ as $n\rightarrow\infty$.
\end{proposition}

\begin{proof}
  Apply Theorem \ref{sparseturannumber1} with $v=e-q-1$.
  Since for every $1\le i\le e$,
  \begin{equation*}
    \begin{aligned}
      &ir-\lceil\fr{(i-1)(er-e+q+1)}{e-1}\rceil=ir-(i-1)(r-1)-\lceil\fr{(i-1)(r+q)}{e-1}\rceil\ge i-q-1,\\
    \end{aligned}
  \end{equation*}
so there exists an $r$-graph with $\Omega(n^{\fr{er-e+q+1}{e-1}}(\log n)^{\fr{1}{e-1}})$ edges, which is  $\mr{G}_r(i-q-1,i)$-free for every $1\le i\le e$, as needed.
\end{proof}

\subsection{$r$-graphs with no short Berge cycles}

\noindent For integers $t\ge 3,r\ge 3$, a {\it Berge $t$-cycle} in an $r$-graph is a set of $t$ distinct vertices $v_1,\ldots,v_t$ associated with $t$ distinct\footnote{In the literature, some authors (see, e.g. \cite{Lazebnikgirth5}) require the edges in a Berge cycle to be distinct, while others (see, e.g. \cite{Jacquesbergesurvey}) do not. However, it is easy to show that if there are at least two distinct edges in the cycle, then a Berge cycle {\it without} distinctness contains a Berge cycle {\it with} distinctness. Since in this paper we only consider the length of a shortest Berge cycle, the definition with distinctness is more suitable for us.}
edges $A_1,\ldots,A_t$ such that $\{v_{i-1},v_i\}\s A_i$ for $2\le i\le t$ and $\{v_1,v_t\}\s A_1$.
An $r$-graph is said to be {\it $\mr{B}_t$-free} if it contains no Berge cycles of length at most $t$.
For $t=3$, the results of \cite{Ruzsa-Szemeredi,EPR} implied that for any $r\ge3$, $n^{2-o(1)}<\e_r(n,\mr{B}_3)=o(n^2).$
For $t=4$, it was shown in \cite{Lazebnikgirth5} that for $r=3$, $\e_3(n,\mr{B}_4)=\Theta(n^{\fr{3}{2}})$,
and in \cite{TimmonsGirth4} that for any $r\ge 4$, $\e_r(n,\mr{B}_4)>n^{\fr{3}{2}-o(1)}.$
Recently, Xing and Yuan \cite{LRC} used $\mr{B}_t$-free $r$-graphs to construct optimal Locally Recoverable
Codes (see Subsection \ref{application-LRC} below) and they showed that (using the alteration method) for any $r\ge 3$ and $t\ge 5$,
$$\e_r(n,\mr{B}_t)=\Omega(n^{\fr{t}{t-1}}).$$
It is not hard to verify (see, e.g. Theorem 5.1 in \cite{LRC}) that an $r$-graph $\ma{H}$ is $\mr{B}_t$-free if and only if it is simultaneously $\mr{G}_r(ir-i,i)$-free for every $1\le i\le t$. Thus applying Theorem \ref{sparseturannumber1} with $v=tr-t$ and $e=t$ leads to the following result.

\begin{proposition}\label{universeconsequence2}
  For fixed integers $r\ge 3,t\ge 5$,
  $$\e_r(n,\mr{B}_t)=\Omega(n^{\fr{t}{t-1}}(\log n)^{\fr{1}{t-1}})$$
  as $n\rightarrow\infty$; or equivalently, there exists an $r$-graph with such number of edges, which is simultaneously $\mr{G}_r(ir-i,i)$-free for every $1\le i\le t$.
\end{proposition}

We remark that in \cite{LRC} the authors stated that in a private communication, Jacques Verstra\"ete suggested that a lower bound on $\e_r(n,\mr{B}_t)$, which is exactly the same with Proposition \ref{universeconsequence2}, can also be proved by using the method of \cite{BohmanKeevashr=2,BohmanKeevashr=3} (which is rather involved).
Nevertheless, since \cite{LRC} stated this result (as well as Proposition \ref{xing-yuan} below) without a proof, we present it here as an easy consequence of Theorem \ref{sparseturannumber1}.

\section{Applications to coding theory}\label{application-2}

\noindent In this section we present three applications of Theorem \ref{sparseturannumber1} to coding theory.
\subsection{Parent-Identifying Set Systems}

\noindent
An $r$-graph $\ma{H}\s\binom{[n]}{r}$ is said to be a {\it $t$-Parent-Identifying Set System} ($t$-IPPS for short),
denoted as $t$-${\rm{IPPS}}(r,|\ma{H}|,n)$, if for any $r$-subset $X\s[n]$ which is contained in the union of at most $t$ edges of $\ma{H}$, it holds that
$$\cap_{\ma{P}\in P_t(X)}\ma{P}\neq\emptyset,$$ where $P_t(X)=\{\ma{P}\s\ma{H}:|\ma{P}|\le t,~X\s\cup_{A\in\ma{P}}A\}$. 

IPPSs were introduced by Collins \cite{Collins09} as a technique to trace traitors in a secret sharing scheme.
Generally speaking, an {\it $(n,r)$-threshold secret sharing scheme}
has one {\it message} and $n$ {\it keys} 
such that any set of at least $r$ keys can be used to decrypt this message but no set of fewer than $r$ keys can. Let $\ma{H}$ be a $t$-${\rm{IPPS}}(r,m,n)$ whose $n$ vertices and $m$ edges are indexed by the $n$ keys and the $m$ users, respectively.
Assume that there is a {\it data supplier} distributes the keys to the users such that
for $1\le i\le m$, the $i$th user 
gets
the $r$ keys 
which form the $i$th edge of $\ma{H}$.
Suppose a coalition of at most $t$ illegal users may collude by combining some of their keys to produce a new, unauthorized set $T$ of $r$ keys to decrypt this message. Then, by definition of a $t$-IPPS, upon capturing an unauthorized set $T$,
the data supplier is able to identify at least one illegal user who contributed to $T$.


For a $t$-${\rm{IPPS}}(r,m,n)$ with given $t,r$ and $n$, it was shown
by Gu and Miao \cite{Gu2018} that $$m=O(n^{\lc\fr{r}{\lf t^2/4\rf+t}\rc}).$$
Recently, Gu, Cheng, Kabatiansky and Miao \cite{Gu} showed that for fixed integers $t\ge 2,r\ge 3$, there exists a $t$-${\rm{IPPS}}(r,m,n)$ with
$$m=\Omega(n^{\fr{r}{\lf t^2/4\rf+t}}),$$ 
which implies that for $\lf t^2/4\rf+t\mid r$ the
upper bound in \cite{Gu2018} is tight up to a constant factor.
We slightly improve the lower bound of \cite{Gu} for some pairs of $r,t$.

\begin{proposition}\label{ipps}
  For fixed integers $t\ge 2,r\ge 3$ satisfying $\gcd(\lf t^2/4\rf+t,r)=1$, there exists a $t$-${\rm{IPPS}}(r,m,n)$ with
  $$m=\Omega(n^{\fr{r}{\lf t^2/4\rf+t}}(\log n)^{\fr{1}{\lf t^2/4\rf+t}})$$ as $n\rightarrow\infty$.
\end{proposition}

Proposition \ref{ipps} is proved by establishing a connection between IPPSs and sparse hypergraphs, as stated below.
Note that a similar observation with different phrasing was obtained independently in \cite{Gu}.

\begin{lemma}\label{link}
  Assume that $\ma{H}\s\binom{[n]}{r}$ is a $\mr{G}_r(er-r,e)$-free $r$-graph with $e=\lf(t/2+1)^2\rf$. Then it is also
  a $t$-${\rm{IPPS}}(r,|\ma{H}|,n)$.
\end{lemma}


\begin{proof}
\noindent  Assume towards contradiction that $\ma{H}$ is not a $t$-${\rm{IPPS}}(r,|\ma{H}|,n)$. Thus
  by definition there exists an $r$-subset $X\s[n]$, which can be covered by at most $t$ edges of $\ma{H}$, such that $\cap_{\ma{P}\in P_t(X)}\ma{P}=\emptyset$.
  Let $m$ be the minimal positive integer such that there exist $\ma{P}_1,\ldots,\ma{P}_m\in\ma{P}_t(X)$ with $\cap_{i=1}^m \ma{P}_i=\emptyset$.
  By the minimality of $m$, it holds that for each $i\in[m]$, $$\cap_{j\in[m]\setminus\{i\}} \ma{P}_j\neq\emptyset.$$
  Without loss of generality, assume $A_i\in\cap_{j\in[m]\setminus\{i\}} \ma{P}_j$.
  Clearly, $A_i\not\in\ma{P}_i$ and moreover, for $1\le i\neq i'\le m$, $A_i\neq A_{i'}$.
  Let $\ma{A}:=\{A_1,\ldots,A_m\}$, then
  \begin{equation*}
    \begin{aligned}
      |\cup_{i=1}^m\ma{P}_i|=&|\ma{A}|+|\cup_{i=1}^m(\ma{P}_i\setminus\ma{A})|\\
      \le&m+\sum_{i=1}^m|\ma{P}_i\setminus\ma{A}|\\
      \le&m+m(t-m+1)\\
      \le&\lf(t/2+1)^2\rf,\\
    \end{aligned}
  \end{equation*}
  where the second inequality follows since    $|\ma{P}_i|\le t$ and $|\ma{P}_i\cap\ma{A}|=m-1$ for any $i\in[m]$.

  Let $\ma{B}:=\cup_{i=1}^m\ma{P}_i$. We claim that for each $x\in X$ there exist at least two distinct sets $B_i,B_j\in\ma{B}$ that contain it.
  Assume the opposite, then there exist $x\in X$ and $B\in\ma{B}$, such that $x$ belongs solely to $B$ but to no other set in $\ma{B}$; that is,
  $x\in B$ and $x\not\in\cup_{B'\in\ma{B}\setminus\{B\}}B'$.  This implies that  $B \in\ma{P}_i$ for any $i$, and  $\cap_{i=1}^m\ma{P}_i\neq \emptyset$, a contradiction.

Add to $\ma{B}$ arbitrary $\lf(t/2+1)^2\rf-|\ma{B}|\geq 0$ edges of $\ma{H}\setminus\ma{B}$. It is clear now that $\ma{B}$ contains exactly $\lf(t/2+1)^2\rf$ edges of $\ma{H}$ and

    \begin{equation*}
    \begin{aligned}
      |\cup_{B\in\ma{B}}B|\leq \lf(t/2+1)^2\rf r -r,
    \end{aligned}
  \end{equation*}
  where the inequality follows since each element of $X$ appears in at least two edges of $\ma{B}$.
	This violates the $\mr{G}_r(er-r,e)$-freeness of $\ma{H}$ for $e=\lf(t/2+1)^2\rf$, and the result follows.
\end{proof}

\begin{proof}[\textbf{Proof of Proposition \ref{ipps}}]
  Apply Lemma \ref{link} and Theorem \ref{sparseturannumber1} with $v=er-r$ and $e=\lf(t/2+1)^2\rf$.
\end{proof}

\subsection{Uniform Combinatorial Batch Codes}\label{application-CBC}

\noindent
An {\it $r$-uniform CBC} with parameters $m,e,n$, denoted as $r{\text{-}}(m,e,n)$-CBC, is an $r$-uniform {\it multihypergraph} (i.e., hypergraphs allowing repeated edges) $\ma{H}$ with $n$ vertices and $m$ edges, such that for every $1\le i\le e$, the union of any $i$ distinct edges contains at least $i$ vertices.
For integers $e>r\ge 3$, let $m(n,r,e)$ denote the maximum $m$ such that an $r{\text{-}}(m,e,n)$-CBC exists.

Uniform-CBCs can be applied to
the following scenario in a distributed database system, as illustrated by Balachandran and Bhattacharya \cite{batch1}. Assume that there are $m$
{\it data items} which are stored in $n$ {\it servers} and any data item is replicated across $r$
servers so that any $e$ of the $m$ data items can be retrieved by accessing $e$ servers and reading  exactly one data item from each.
Let $\ma{H}\s\binom{[n]}{r}$ be an $r$-uniform multihypergraph whose $n$ vertices and $m$ edges are indexed by the servers and the data items, respectively. An {\it $\ma{H}$-based replication system} stores $m$ data items among $n$ servers as follows: 
for $1\le i\le m$, the $i$th data item is stored in the $r$ servers which form the $i$th edge of $\ma{H}$\footnote{Since two distinct data items may be stored in the same set of $r$ servers, $\ma{H}$ is
allowed to have repeated edges.}.

Given an $\ma{H}$-based replication system,
the required retrieval condition on the servers and the data items can be expressed in a purely combinatorial way: every collection of at most $e$ distinct edges of $\ma{H}$ has a {\it system of distinct representatives} (SDR for short) from the $n$ vertices, where for any $e$ edges $\ma{A}=\{A_1,\ldots,A_e\}\s\ma{H}$, an SDR of $\ma{A}$ is a set of $e$ {\it distinct} elements $\{x_1,\ldots,x_e\}\s[n]$ such that $x_i\in A_i$ for each $1\le i\le e$.
Applying Hall's theorem \cite{hall1935representatives} one can infer that this holds if and only if $\ma{H}$ is $\mr{G}_r(i-1,i)$-free for every $1\le i\le e$.

Recall that an $r$-graph is said to be $\mr{H}_r(q,e)$-free if it is simultaneously $\mr{G}_r(i-q-1,i)$-free for every $1\le i\le e$.
Clearly, an $r{\text{-}}(m,e,n)$-CBC is equivalent to an $\mr{H}_r(0,e)$-free $r$-uniform multihypergraph with $n$ vertices and $m$ edges; consequently,

\begin{equation*}\label{batch}
  \begin{aligned}
    m(n,r,e)\ge \e_r(n,\mr{H}_r(0,e)).
  \end{aligned}
\end{equation*}

\noindent For fixed integers $e> r\ge 3$, it was shown in \cite{paterson2009} that
$$m(n,r,e)=\Omega(n^{r-1+\fr{r}{e-1}}).$$

An easy application of Proposition \ref{universeconsequence1} suggests the following result.

\begin{proposition}\label{cbc}
For fixed integers $e> r\ge 3$ satisfying $\gcd(e-1,r)=1$, 
  $$m(n,r,e)\ge \e_r(n,\mr{H}_r(0,e))=\Omega(n^{r-1+\fr{r}{e-1}}(\log n)^{\fr{1}{e-1}})$$ as $n\rightarrow\infty$.
\end{proposition}

\subsection{Optimal Locally Recoverable Codes}\label{application-LRC}
\noindent
A {\it linear code} $\ma{C}$ of length $n$ defined on the finite field $\mathbb{F}_q$ is a subspace of  $\mathbb{F}_q^n$. The {\it minimum distance} of $\ma{C}$ is defined as $d(\ma{C}):=\min\{{\rm{wt}}(\pmb{x}):\pmb{x}\in\ma{C}\setminus\{\pmb{0}\}\}$, where ${\rm{wt}}(\pmb{x})$ is the number of nonzero coordinates of $\pmb{x}$. A {\it parity check matrix} of $\ma{C}$ is an $(n-k)\times n$ matrix $H$ such that $\pmb{x}\in\ma{C}$ if only if $H\cdot \pmb{x}^T=\pmb{0}.$

A linear code $\ma{C}\subseteq\mathbb{F}_q^n$ of dimension $k$ is called {\it Locally Recoverable Code} (or LRC for short) with {\it locality} $r$ if for any $i\in[n]$ there exists  $r$ other coordinates $i_1,...,i_r$ such that for any codeword $\pmb{x}=(x_1,\ldots,x_n)\in\ma{C}$,  $x_i$ can be recovered from   $x_{i_1},\ldots,x_{i_r}$.
We denote such a code by $(n,k,r)$-LRC. In \cite{LocalSingleton} it was shown that the minimum distance $d$ of an $(n,k,r)$-LRC satisfies $d\le n-k-\lc\fr{k}{r}\rc+2,$ and the code is called {\it optimal} if the bound is achieved with equality.


In order to reduce the complexity of the operations in the  finite field, it is desirable to define the LRCs over small enough fields. In other words, given the size of the underlying field, our goal is to construct the longest possible optimal-LRC.

Assume that $r+1\mid n$. Set $m:=\frac{n}{r+1}$ and let $I_m$  and $\textbf{1}$ be the identity matrix of order $m$, and the all $1$ row vector of length $r+1$, respectively. It is not hard to verify that a linear code $ \ma{C}$ with   parity check matrix of the form
\begin{equation}
\label{eq:8}
H=\left(\begin{matrix}
I_m\otimes \textbf{1}\\
A\\
\end{matrix}\right),
\end{equation}
\noindent where $ \otimes$ is the Kronecker product and $A$ is an $(n-k-m)\times n$ matrix, has locality $r$. Indeed, any symbol $x_i,i\in [n]$ of a codeoword $\pmb{x}\in\ma{C}$  can be recovered by $r$ other symbols since it satisfies a linear equation which has exactly  $r+1$ variables.

%
Xing and Yuan \cite{LRC} gave a construction of an optimal LRC for $r\ge d-2$ by carefully constructing the matrix $A$ in \eqref{eq:8}, as follows.
%
For a subset $A=\{\alpha_1,...,\alpha_{r+1}\}\subseteq \mathbb{F}_q$, let $V(A)$ be the $(d-2)\times (r+1)$ Vandermonde matrix with $\alpha_j^i$ as its $(i,j)$-entry.
The following result  was proved in \cite{LRC}.
\begin{lemma}[see Theorem 3.1, \cite{LRC}]\label{LRCsparse}
Let  $d\ge 11$ and $r\ge d-2$, and  let $\ma{C}\s\mathbb{F}_q^n$ be a linear code with parity check matrix
\begin{equation*}
H=\left(\begin{matrix}
I_m\otimes \textbf{1}\\
V(A_1),...,V(A_m)\\
\end{matrix}\right),
\end{equation*}
%
%
 then $\ma{C}$ is an optimal $(n,k,r)$-LRC with minimum distance $d$ if and only if the family $\ma{A}:=\{A_1,\ldots,A_m\}\s\binom{\mathbb{F}_q}{r+1}$ is $\mr{G}_{r+1}(ir,i)$-free for each $1\le i\le \lf\fr{d-1}{2}\rf$.
\end{lemma}
The following result (which is stated in \cite{LRC} without a  proof) follows by combining Lemma \ref{LRCsparse} and Proposition \ref{universeconsequence2}.

\begin{proposition}[see also Theorem 1.1, \cite{LRC}]\label{xing-yuan}
  Suppose that $d\ge 11$, $r\ge d-2$ and $r+1\mid n$, then there exists an optimal $(n,k,r)$-LRC over $\mathbb{F}_q$ with minimum distance $d$ and length $n=\Omega\big(q(q\log q )^{\fr{1}{\lf(d-3)/2\rf}}\big).$
\end{proposition}

\section*{Acknowledgements}

\noindent The research of Chong Shangguan and Itzhak Tamo was supported by ISF grant No. 1030/15 and NSF-BSF grant No. 2015814.
The authors would like to thank Prof. Yiwei Zhang for valuable comments on the first version of this manuscript.
They are also grateful to Yujie Gu for helpful discussions on Parent-Identifying Set Systems. Lastly, the authors want to express their gratitude to the two anonymous reviewers for their comments which are very helpful to the improvement of this paper.

{\small
\bibliographystyle{plain}
\bibliography{turan}
}
 \end{document}